\documentclass{amsart}

\usepackage{amssymb, amsmath}
\usepackage{mathrsfs}
\usepackage{amscd}
\usepackage{verbatim}
\usepackage{wasysym}
\usepackage[usenames,dvipsnames]{xcolor}

\allowdisplaybreaks

\usepackage{tikz}
\usepackage{pgfplots}

\usepackage{marvosym}

\usepackage{rotating}

\usepackage[utf8]{inputenc}
\usepackage{enumerate}
\usepackage{url}

\usepackage[colorinlistoftodos,prependcaption,textsize=tiny]{todonotes}

\usepackage[colorlinks,linkcolor={blue},citecolor={blue},urlcolor={red},]{hyperref}

\theoremstyle{plain}
\newtheorem{theorem}{Theorem}
\theoremstyle{remark}
\newtheorem{remark}[theorem]{Remark}

\newtheorem{example}[theorem]{Example}

\theoremstyle{plain}

\newtheorem{lemma}[theorem]{Lemma}
\newtheorem{proposition}[theorem]{Proposition}

\numberwithin{theorem}{section}
\numberwithin{pro}{section}

\numberwithin{equation}{section}

\numberwithin{equation}{section}

\def\N{{\mathbb N}}

\def\R{{\mathbb R}}
\def\C{{\mathbb C}}

\newcommand{\E}{{\mathbb E}}
\renewcommand{\P}{{\mathbb P}}
\newcommand{\A}{{\mathcal A}}

\newcommand{\F}{{\mathcal F}}

\newcommand{\Schw}{{\mathscr S}}

\renewcommand{\Re}{\hbox{\rm Re}\,}

\newcommand{\calL}{L}

\newcommand{\one}{{{\bf 1}}}

\newcommand{\wh}{\widehat}

\newcommand {\ud}{\, \mathrm{d}}

\begin{document}

\author{Martin Ondrej\'at}
\address{The Czech Academy of Sciences \\ Institute of Information Theory and Automation \\ Pod Vod\'arenskou v\v e\v z\'{\i} 4 \\ 182 08 Prague 8\\ Czech Republic} \email{ondrejat@utia.cas.cz}

\author{Mark Veraar}
\address{Delft Institute of Applied Mathematics\\
Delft University of Technology \\ P.O. Box 5031\\ 2600 GA Delft\\The
Netherlands} \email{M.C.Veraar@tudelft.nl}

\thanks{The research of the first named author was supported by the Czech Science Foundation grant no. 19-07140S}
\thanks{The second named author is supported by the VIDI subsidy 639.032.427 of the Netherlands Organisation for Scientific Research (NWO)}

\title[Temporal regularity of stochastic convolutions]{On temporal regularity of stochastic convolutions in $2$-smooth Banach spaces}


%\date{\today}

\begin{abstract}
We show that paths of solutions to parabolic stochastic differential equations have the same regularity in time as the Wiener process (as of the current state of art). The temporal regularity is considered in the Besov-Orlicz space $B^{1/2}_{\Phi_2,\infty}(0,T;X)$ where $\Phi_2(x)=\exp(x^2)-1$ and $X$ is a $2$-smooth Banach space.

\bigskip

\noindent {\tiny SUMMARY.} Nous montrons que les trajectoires des solutions des \'equations aux derive\'es partielles stochastiques paraboliques ont la m\^eme r\'egularit\'e en temps que le processus de Wiener (\`a la pointe de la connaissance actuelle). La r\'egularit\'e temporelle est consid\'er\'ee dans l'espace de Besov-Orlicz $B^{1/2}_{\Phi_2,\infty}(0,T;X)$ o\`u $\Phi_2(x)=\exp(x^2)-1$ et $X$ est un espace de Banach $2$-lisse.

\end{abstract}

\maketitle

\section{Introduction}

It is well known that paths of the Brownian motion $W$ belong to the H\"older spaces $\mathcal C^\alpha([0,T])$ for $\alpha<1/2$ a.s.\ but $\Bbb P\,(W\in\mathcal C^\alpha([0,T]))=0$ for $\alpha\ge 1/2$. Zbigniew Ciesielski showed in \cite{ciesielski1991modulus} that one can obtain smoothness of order $1/2$ in the Besov spaces $B^{\frac12}_{p,\infty}(0,T)$ for $p<\infty$, and later on, Bernard Roynette proved in \cite{Roy} that this is actually the best regularity in the scale of the Besov spaces one can get, i.e. that the brownian sample paths are in the class of Besov spaces $B^\alpha_{p,q}(0,T)$ a.s. if and only if $\alpha<1/2$, or $\alpha=1/2$, $p<\infty$ and $q=\infty$. The H\"older spaces are particular cases of Besov spaces as $\mathcal C^\alpha=B^\alpha_{\infty,\infty}$ and $B^\alpha_{p,\infty}\subseteq C^{\alpha-\frac 1p}$ for $\alpha\in(0,1)$ and $p\in[1,\infty]$ e.g. by \cite{Simon}.
It follows that there is no smallest H\"older space or Besov space to which brownian paths belong to almost surely. However, if one allows for more general H\"older spaces - so called modulus H\"older spaces (that generalize the class of the standard H\"older spaces), i.e. $f\in \mathcal C^\varphi([0,T])$ if and only if
$$
|f(t)-f(s)|\le c\varphi(|t-s|)\text{ for all }s,t\in[0,T]\text{ and some finite constant }c,
$$
then one can get to the end - to the smallest space in this class with the desired property. Namely, Paul L\'evy showed in \cite{levy_37} that almost all paths of $W$ belong to the modulus H\"older space $\mathcal C^g([0,T])$ with $g(r):=|r\log r|^{1/2}$ for $r$ small and this space is the best (i.e. smallest) among all modulus H\"older spaces $\mathcal C^\varphi([0,T])$ with this property, see e.g. \cite[Theorem 1.1.1]{csrev_81}.

The remaining problem was that the Besov spaces $B^{1/2}_{p,\infty}(0,T)$ for $p<\infty$ and the L\'evy space $\mathcal C^g([0,T])$ are not included one in another, see e.g. \cite{CKR_1993}, so the smallest space containing almost all brownian paths was still missing.

In 1993, Zbigniew Ciesielski found a function space that is contained both in $B^{1/2}_{p,\infty}(0,T)$ for all $p<\infty$ and in the L\'evy space $\mathcal C^g([0,T])$, see e.g. \cite{CKR_1993}, and almost all paths of $W$ belong to it, see \cite{CI_1993}. It is the Besov-Orlicz space $B^{1/2}_{\Phi_2,\infty}(0,T)$ where $\Phi_2(x)=\exp(x^2)-1$. This result was later generalized in \cite{HytVer}, using a different method of proof, to cover also Wiener processes with values in a Banach space $X$.

There are several papers in which precise Besov regularity of other stochastic processes than Brownian motion are studied: fractional Brownian motion \cite{CKR_1993}, \cite{Vercor}), $d$-dimensional white noise \cite{Veraarwhitenoise}), L\'evy noise \cite{aziznejad2018wavelet}, \cite{FFU}, \cite{FUW}, \cite{Schilling97}. The paper \cite{BeOh} studies optimal path regularity of periodic Brownian motion in modulation spaces, Wiener amalgam spaces, Fourier–Lebesgue spaces and Fourier–Besov spaces on the torus.

In \cite{OSK}, it was proved that not only the Wiener process has paths in $B^{1/2}_{\Phi_2,\infty}(0,T)$ almost surely but that the same holds true for all continuous local martingales with Lipschitz continuous quadratic variation. And, moreover,  that there is a continuity property in the sense that convergence in probability of the quadratic variations in the Lipschitz norm yields convergence in probability of the continuous local martingales in the norm of $B^{1/2}_{\Phi_2,\infty}(0,T)$. Consequently, paths of solutions to stochastic differential equations with locally bounded non-linearities belong to the space $B^{1/2}_{\Phi_2,\infty}(0,T)$ almost surely.

Unfortunately, the idea of the proof in \cite{OSK} was based on a change-of-time argument and therefore it was not applicable to infinite-dimensional martingales and SPDEs. In this paper, we overcome this drawback and we generalize the results in \cite{OSK} and \cite{HytVer} not only to infinite-dimensional stochastic integrals but also to stochastic convolutions in $2$-smooth Banach spaces, and, consequently, we show that paths of mild solutions to parabolic stochastic differential equations in any $2$-smooth Banach space $X$ have paths in the Besov-Orlicz space $B^{1/2}_{\Phi_2,\infty}(0,T;X)$ almost surely. Let us recall that e.g. $L^p$, $W^{s,p}$, $B^\alpha_{p,q}$, $F^\alpha_{p,q}$ are $2$-smooth for $p,q\in[2,\infty)$, $s>0$, $\alpha\in\Bbb R$. Our main result is as follows and is already new in the Hilbert space setting. More details on the function spaces can be found in Section \ref{sec:prel} and details on stochastic integration and convolutions can be found in Sections \ref{sec:Stochint} and \ref{sec:SPDE}, respectively.

\begin{theorem}\label{thm:mainintro}
Let $X$ be a separable $2$-smooth Banach space and let $A$ be the generator of an analytic $C_0$-semigroup on $X$.
Let $f\in L^0_{\F}(\Omega;L^\infty(0,T;\gamma(H,X)))$. Let $H$ be a separable Hilbert space and $W$ be an $H$-cylindrical Brownian motion. Then the mild solution $u$ of the problem:
\begin{equation*}
d u = A u \ud t + f \ud W, \ \ u(0) = 0
\end{equation*}
satisfies $u\in B^{1/2}_{\Phi_2,\infty}(0,T;X)$ a.s. and the corresponding solution mapping $f\mapsto u_f$
is continuous in the following sense:
$$
\|f_n-f\|_{L^\infty(0,T;\gamma(H,X))}\xrightarrow{\Bbb P} 0\quad\implies\quad\|u_{f_n}-u_f\|_{B^{1/2}_{\Phi_2,\infty}(0,T;X)}\xrightarrow{\Bbb P} 0.
$$
\end{theorem}
Theorem \ref{thm:mainintro} will be proved in Section \ref{sec:SPDE} where a more general result will be discussed as well. At first sight the condition on $f$ seems quite special but, typically, $f$ is of the form $f = B(v)$, where $v$ is the solution to an SPDE (for instance $v = u)$ and where $B$ is a Lipschitz function on $X$. In this case one usually has $v\in L^0(\Omega;C([0,T];X))$ and $f = B(v)$ indeed satisfies the required condition.

The class of $2$-smooth Banach spaces plays an important role in stochastic analysis in infinite dimensions. For instance for this class of spaces one can obtain exponential estimates for discrete martingales (see \cite{Pin})) and sharp maximal inequalities for stochastic integrals and convolutions (see \cite{BrzPes}, \cite{NZ} and \cite{Seid10}). It is an open problem whether there is an extension of the results of this paper to the class of UMD Banach spaces $X$ (see \cite{NVW15}). In particular, motivated by \cite{OSK} it would be interesting to obtain an analogue of Theorem \ref{thm:mainindef} below for $X$-valued continuous local martingales with a suitable quadratic variation. Note that recently the existence of such a quadratic variation was established in \cite{Yar2018}.

Temporal regularity in H\"older spaces $C^\alpha([0,T];X)$ or Besov spaces $B^\alpha_{p,q}(0,T;X)$ for solutions to SPDEs driven by Wiener processes was established so far only for $\alpha<1/2$. It is not possible to list all relevant papers here so we refer the reader just to some of them, e.g. \cite{DPKZ}, \cite{B_1997}, \cite{PS_1998}, \cite{MM_2000}, \cite{SSVui_2002}, \cite{SSVui_2003}, \cite{BN_2003},  \cite{DalSS_2005}, \cite{NVW_2012_a}, \cite{NVW_2012_b}, \cite{CKLF_2013}, \cite{NVW_2015}, \cite{AKM_2016}.

{\em Acknowledgment:}
The authors thank the referees for a careful study of the paper and for their comments and recommendations.

\section{Preliminaries}\label{sec:prel}

For the theory of vector-valued function spaces used in this paper we refer the reader to \cite{Am97, Ama09, HNVW1,Pick_Sickel, Schm,SchmSiunpublished, SchmSi05, Tr97} and references therein.

\subsection{Orlicz spaces\label{subsec:Orlicz}}

For extensive treatments of the theory of Orlicz function we refer to \cite{RaoRen, Zaa}.

Let $X$ be a Banach space, $\mathcal N$ a Young function, i.e. a non-negative, non-decreasing, left-continuous, convex function on $[0,\infty)$ such that $\mathcal N(0)=0$, $\mathcal N(\infty-)=\infty$ and let $(U,\mathcal U,\mu)$ be a $\sigma$-finite measure space. Then, for a Bochner measurable function $f:U\to X$, we define the Luxemburg norm
$$
\|f\|_{L^{\mathcal N}(U,\mu)}=\inf\,\left\{\lambda>0:\int_U\mathcal N(\|f\|_X/\lambda)\ud \mu\le 1\right\}
$$
and the Banach space $L^{\mathcal N}(U,\mu;X)=\{f:\|f\|_{L^{\mathcal N}(U,\mu;X)}<\infty\}$ equipped with the Luxemburg norm is called the Orlicz space with the Young function $\mathcal N$ (sometimes it is called an $N$-function). If there is no confusion, we will write shortly just $\|\cdot\|_{\mathcal N}$ instead of $\|\cdot\|_{L^{\mathcal N}(U,\mu;X)}$.

Orlicz spaces can be introduced alternatively and equivalently via the norm being the middle term in the formula \eqref{eq:lux_norm} below:
\begin{equation}\label{eq:lux_norm}
\|f\|_{\mathcal N}\le\inf_{\lambda>0}\frac{1}{\lambda}\left[1+\int_U\mathcal N(\lambda\|f\|_X)\ud \mu\right]\le 2\|f\|_{\mathcal N}
\end{equation}
for every Bochner measurable function $f:U\to X$, see \cite[Lemma 2.1]{HytVer}.

\begin{example}[scaling] Let $G$ and $U$ be open sets in $\Bbb R^d$, let $g:G\to U$ be a diffeomorphism such that $a\le|\operatorname{det}g^\prime|\le b$ on $G$ for some positive constants $a$, $b$. Then, by convexity of $\mathcal N$,
\begin{equation}\label{eq:scaling}
\min\,\{a,1\}\|f\circ g\|_{L^{\mathcal N}(G;X)}\le\|f\|_{L^{\mathcal N}(g[G];X)}\le\max\,\{b,1\}\|f\circ g\|_{L^{\mathcal N}(G;X)}
\end{equation}
holds for every Bochner measurable function $f:g[G]\to X$.
\end{example}

\subsection{Besov-Orlicz spaces on $\Bbb R$\label{subsec:BesovI}}

Let $\Schw(\R;X)$ denote the $X$-valued Schwartz functions. Let $\Schw'(\R;X) = \calL(\Schw(\R), X)$ denote the space of vector-valued tempered distributions.

Fix $\varphi \in \Schw(\R)$  such that
\begin{equation}\label{eq:propvarphi}
0\leq \wh{\varphi}(\xi)\leq 1, \quad  \xi\in \R, \qquad  \wh{\varphi}(\xi) = 1 \ \text{ if } \ |\xi|\leq 1, \qquad  \wh{\varphi}(\xi)=0 \ \text{ if } \ |\xi|\geq \frac32.
\end{equation}
Let $\wh{\varphi}_0 = \wh{\varphi}$, $\wh{\varphi}_1(\xi) = \wh{\varphi}(\xi/2) - \wh{\varphi}(\xi)$ and
\[\wh{\varphi}_k(\xi) = \wh{\varphi}_1(2^{-k+1} \xi) = \wh{\varphi}(2^{-k}\xi) - \wh{\varphi}(2^{-k+1}\xi),  \qquad \xi\in \R, \qquad  k\geq 1.\]

Fix also $\phi\in\Schw(\R)$ such that the support of $\wh\phi$ is contained in the set $[\frac 12<|\xi|<2]$,
\begin{equation}\label{eq:propphi}
\sum_{k\in\Bbb Z}\wh\phi(2^{-k}\xi)=1\quad\text{for}\quad\xi\ne 0,
\end{equation}
and define $\phi_j(x)=2^j\phi(2^j\xi)$ for $x\in\Bbb R$ and $j\in\Bbb Z$.

For a Banach space $X$, a Young function (see Section \ref{eq:lux_norm} for the definition) $\mathcal N$, $q\in [1,\infty]$, and $s\in\R$ the {\em Besov-Orlicz space} $B_{\mathcal N,q}^s(\R;X)$ is defined as the space of all $f\in {\mathscr S}'(\R;X)$ for which
\[ |f|_{B_{\mathcal N,q}^s (\R;X)} := \Big\| \big( 2^{ks}\varphi_k* f\big)_{k\geq 0} \Big\|_{\ell^q(L^{\mathcal N}(\R;X))} < \infty.\]
This defines a Banach space. One can check that if one uses a different function $\varphi$, this leads to the same space with an equivalent norm.

We also define the {\em homogeneous Besov-Orlicz space} $\dot B_{\mathcal N,q}^s(\R;X)$ as the space of all $f\in {\mathscr S}'(\R;X)$ for which
\[ |f|_{\dot B_{\mathcal N,q}^s (\R;X)} := \Big\| \big( 2^{js}\phi_j* f\big)_{j\in\Bbb Z} \Big\|_{\ell^q(L^{\mathcal N}(\R;X))} < \infty.\]
This defines a complete pseudonormed space.

We refer the readers to \cite{Pick_Sickel} on basic properties of real-valued Besov-Orlicz spaces and to \cite{BeLo} for real-valued homogeneous Besov spaces. Both vector-valued spaces do not differ significantly from their real-valued counterparts, as observed already in \cite{Pick_Sickel}.

\begin{remark}
$B_{\mathcal N,q}^s (\R;X)$ is the standard Besov space $B_{p,q}^s (\R;X)$ if $\mathcal N(t)=t^p$ for $p\in[1,\infty)$.
\end{remark}

\subsection{Besov-Orlicz spaces on intervals} In this section, we introduce Besov-Orlicz spaces on intervals $I\subseteq\Bbb R$ and we will show that if $I=\Bbb R$ then the norms here and in Section \ref{subsec:BesovI} are equivalent. Next we introduce several equivalent norms, we construct an extension operator and finally we show how the spaces change under scalings. For the purposes of the paper, it is important that the constants in \eqref{eq:scaling}, \eqref{eq:eq_fourier}, \eqref{eq:Johnen}, \eqref{eq_dyadic} and \eqref{eq:ext_oper} do not depend on $X$ and $\mathcal N$. We refer the readers for details on real-valued Besov spaces to \cite{Tr1} and for the vector-valued Besov spaces to the treatise \cite{Ko}. Below, $X$ is a Banach space, $I$ a bounded or unbounded interval in $\Bbb R$, $\alpha\in (0,1)$, $\mathcal N$ a Young function and $q\in [1, \infty]$.

\subsubsection{Equivalent norms} Let $f:I\to X$ Bochner measurable and define
\begin{align*}
I(h)&=\{s\in I:s+h\in I\}
\\
\Delta_hf(s)&=f(s+h)-f(s)\text{ for } s\in I(h)
\\
\omega_{\mathcal N,I}(f,t)&=\sup\,\{\|\Delta_hf\|_{L^{\mathcal N}(I(h);X)}:|h|\leq t\}
\\
\|f\|_{\mathcal N,I,q, \alpha}&=\|2^{j\alpha} \|\Delta_{2^{-j}}f\|_{L^{\mathcal N}(I(2^{-j});X)}\|_{\ell_q(j\in\Bbb Z)}
\\
K_{\mathcal N,I}(f,t)&=\inf\,\{\|f-g\|_{L^{\mathcal N}(I;X)}+t\|\dot g\|_{L^{\mathcal N}(I;X)}:\,g\in W^{1,1}_{\rm loc}(I;X)\}
\end{align*}
where $W^{1,1}_{\rm loc}(I;X)$ denotes the Sobolev space of functions $g:I\to X$ for which the weak derivatives satisfies $\dot g\in L^1(a,b;X)$ and
$$
g(b)-g(a)=\int_a^b\dot g(s)\ud s
$$
holds for every $[a,b]\subseteq I$. In the next result we allow certain (semi)-norms to be infinite. In this case the result states that both expressions are infinite if one of them is.

\begin{proposition}\label{prop_16} Let $X$ be a Banach space, $\alpha\in(0,1)$, $\mathcal N$ and $\mathcal A$ Young functions such that $\mathcal A>0$ on $(0,\infty)$, $q\in [1, \infty]$ and $I$ a bounded or unbounded interval in $\Bbb R$. Then
\begin{equation}\label{eq:eq_fourier}
c_{\alpha,q,\varphi}^{-1}|f|_{B^\alpha_{\mathcal N,q}(\Bbb R;X)}\le\|f\|_{L^{\mathcal N}(\Bbb R;X)}+\|t^{-\alpha}\omega_{\mathcal N,\Bbb R}(f,t)\|_{L^q(0,\infty;t^{-1} \ud t)}\le c_{\alpha,q,\varphi}|f|_{B^\alpha_{\mathcal N,q}(\Bbb R;X)}
\end{equation}

\begin{equation}\label{eq:eq_fourier_homog}
c_{\alpha,q,\phi}^{-1}|f|_{\dot B^\alpha_{\mathcal A,q}(\Bbb R;X)}\le\|t^{-\alpha}\omega_{\mathcal A,\Bbb R}(f,t)\|_{L^q(0,\infty;t^{-1} \ud t)}\le c_{\alpha,q,\phi}|f|_{\dot B^\alpha_{\mathcal A,q}(\Bbb R;X)}
\end{equation}

\begin{equation}\label{eq:Johnen}
\frac 12\omega_{\mathcal N,I}(g,t)\le K_{\mathcal N,I}(g,t)\le 24\omega_{\mathcal N,I}(g,t),\qquad t>0
\end{equation}

\begin{equation}\label{eq_dyadic}
c_\alpha^{-1}\|g\|_{\mathcal N,I,q,\alpha}\le\|t^{-\alpha}\omega_{\mathcal N,I}(g,t)\|_{L^q(0,\infty;t^{-1} \ud t)}\le c_\alpha\|g\|_{\mathcal N,I,q,\alpha}
\end{equation}
hold for every Bochner measurable functions $f:\Bbb R\to X$ and $g:I\to X$ where the constants $c_{\alpha,q,\varphi}$ and $c_{\alpha,q,\phi}$ depend only on $\alpha$, $q$ and $\varphi$, resp. $\phi$ but not on $X$ or $\mathcal N$ and $c_\alpha$ only on $\alpha$ but not on $X$, $\mathcal N$, $q$ or $I$.
\end{proposition}

\begin{proof}
The above inequalities are known to hold in real-valued Besov and homogeneous Besov spaces and \eqref{eq:eq_fourier} also in real-valued Besov-Orlicz spaces (this observation was already made in \cite{Pick_Sickel}). It is actually a routine to have them proved for vector-valued functions so we content ourselves with references to the real-valued spaces. The estimation \eqref{eq:eq_fourier} can be shown as in \cite[Theorem 1]{Pick_Sickel}, the proof of \eqref{eq:eq_fourier_homog} goes along the same lines as the proof of \cite[Theorem 6.3.1]{BeLo} for real-valued $L^p$-spaces where a straightforward generalization of Lemma \cite[Lemma 1]{Pick_Sickel} to vector-valued spaces is used, \eqref{eq:Johnen} follows by a routine generalization of the result in \cite{Johnen_Scherer} from $L^p$-spaces to Orlicz spaces (see also \cite[Theorem 6.2.4]{DeVore_Lorentz} and \cite[Proposition 3.b.5]{Ko}), and \eqref{eq_dyadic} is based on the same dyadic approximation argument as in \cite[Corollary 3.b.9]{Ko}.
\end{proof}

It is thus consistent with section \ref{subsec:BesovI} to define {\em vector-valued Besov-Orlicz spaces} $B^{\alpha}_{\mathcal N,q}(I;X)$ on intervals $I\subseteq\Bbb R$ as Banach spaces via the norm
$$
\|f\|_{B^{\alpha}_{\mathcal N,q}(I;X)}=\|f\|_{L^{\mathcal N}(I;X)}+\|t^{-\alpha}\omega_{\mathcal N,I}(f,t)\|_{L^q(0,\infty;t^{-1}\ud t)},
$$
that is $B^{\alpha}_{\mathcal N,q}(I;X)=\{f\in L^{\mathcal N}(I;X):\|f\|_{B^{\alpha}_{\mathcal N,q}(I;X)}<\infty\}$.

\begin{remark} One may define, analogously, also vector-valued homogeneous Besov-Orlicz spaces on intervals but such definition does not lead to meaningful objects already in the real-valued case. We need \eqref{eq:eq_fourier_homog} just for technical purposes, see section \ref{subsec:hold}.
\end{remark}

\subsubsection{Extension operators} The inequality \eqref{eq:Johnen} yields that Besov-Orlicz spaces are isomorphic with the real-interpolation spaces between $L^{\mathcal N}(I;X)$ and $W^{\mathcal N,1}(I;X)$ while making obvious that the norms of the isomorphisms can be estimated uniformly with respect to $\alpha\in(0,1)$, $q\in[1,\infty]$, the Young function $\mathcal N$ and the Banach space $X$. Hence, every continuous linear extension operator from $L^{\mathcal N}(I;X)$ to $L^{\mathcal N}(\Bbb R;X)$ which maps $W^{\mathcal N,1}(I;X)$ into $W^{\mathcal N,1}(\Bbb R;X)$ continuously, maps $B^\alpha_{\mathcal N,q}(I;X)$ into $B^\alpha_{\mathcal N,q}(\Bbb R;X)$ continuously. It is therefore easy to see that if the operator $\mathcal E_I$ is defined by reflection at the boundary of $I$ (see \cite[Theorem 5.19]{AdaFou}), the following holds.

\begin{proposition}\label{prop:ext} Let $I$ be a non-trivial bounded or unbounded interval in $\Bbb R$. Then there exists a linear operator $\mathcal E_I$ from the space of $X$-valued Bochner measurable functions on $I$ to $X$-valued Bochner measurable functions on $\Bbb R$ such that $\mathcal E_If=f$ on $I$ and
\begin{equation}\label{eq:ext_oper}
\|\mathcal E_If\|_{B^\alpha_{\mathcal N,q}(\Bbb R;X)}\le\kappa(\alpha,|I|)\|f\|_{B^\alpha_{\mathcal N,q}(I;X)}
\end{equation}
hold for every $f:I\to X$, $\alpha\in(0,1)$, $q\in[1,\infty]$ and every Young function $\mathcal N$ where the constant $\kappa(\alpha,|I|)$ depends only on $\alpha$ and the Lebesgue measure of $I$.
\end{proposition}

\subsubsection{Scaling}
Let $I$ be a non-trivial bounded interval in $\Bbb R$ and consider an affine bijection $g:(0,1)\to I$. Then
\begin{equation}\label{eq:BO_scaling}
\min\,\{|I|^{-1},|I|^\alpha\}\|f\|_{B^\alpha_{\mathcal N,q}(I;X)}\le\|f\circ g\|_{B^\alpha_{\mathcal N,q}(0,1;X)}\le\max\,\{|I|^{-1},|I|^\alpha\}\|f\|_{B^\alpha_{\mathcal N,q}(I;X)}
\end{equation}
holds for every Bochner measurable $f:I\to X$ by \eqref{eq:scaling}. It therefore often suffices to consider problems in the space $B^\alpha_{\mathcal N,q}(0,1;X)$, passing to the original space $B^\alpha_{\mathcal N,q}(I;X)$ by a suitable affine change of time.

\subsection{Embeddings to H\"older spaces}\label{subsec:hold}

Below, $X$ is a Banach space, $I$ a non-trivial bounded or unbounded interval in $\Bbb R$ and $\Phi_\beta(x)=\exp(x^\beta)-1$.

\subsubsection{Embeddings of Besov spaces $B^\alpha_{p,q}(I;X)$}
Let $p,q\in[1,\infty]$, $\frac 1p<\alpha<1$. Then $B^\alpha_{p,q}(I;X)$ is embedded in the H\"older space $C^{\alpha-\frac 1p}(I;X)$ continuously and there exists a constant $C_{\alpha,p}$ such that
\begin{equation}\label{holder_embeding}
\|f(a)-f(b)\|_X\le C_{\alpha,p}|a-b|^{\alpha-\frac 1p}\|t^{-\alpha}\omega_p(f,t)\|_{L^q(0,\infty;t^{-1}\ud t)}
\end{equation}
holds for every $f\in B^\alpha_{p,q}(I;X)$ and every two points $a,b\in I$ of Lebesgue density of $f$.  See e.g. \cite[Corollary 26]{Simon} for a proof.

\subsubsection{Embeddings of Besov-Orlicz spaces $B^\alpha_{\Phi_\beta,q}(I;X)$}
Let $\beta\in[1,\infty)$, $q\in[1,\infty]$, $\alpha\in(0,1)$. Then $B^\alpha_{\Phi_\beta,q}(I;X)$ is embedded in the modular H\"older space $C^{r^\alpha|\log r|^{1/\beta}}(I;X)$ continuously, i.e. there exists a continuous positive non-decreasing function $\zeta:[0,\infty)\to[0,\infty)$ such that
$$
\lim_{t\to 0}\frac{\zeta(t)}{t^\alpha|\log t|^{\frac 1\beta}}=c
$$
for some $c\in(0,\infty)$ and
\begin{equation}\label{mod_holder_embeding}
\|f(a)-f(b)\|_X\le\|t^{-\alpha}\omega_{\Phi_\beta,I}(f,t)\|_{L^q(0,\infty;t^{-1}\ud t)}\zeta(|a-b|)
\end{equation}
holds for every $f\in L^{\Phi_\beta}(I;X)$ and every two points $a,b\in I$ of Lebesgue density of $f$, and
\begin{equation}\label{mod_cb_embeding}
\|f\|_{L^\infty(I;X)}\le c(\alpha,|I|)\|f\|_{B^\alpha_{\Phi_\beta,q}(I;X)}
\end{equation}
holds by definition for every  Bochner measurable $f:\Bbb R\to X$.

\begin{proof}
Because of trivial embeddings of the Besov-Orlicz spaces, it suffices to show \eqref{mod_holder_embeding} for $q=\infty$. And since \eqref{mod_holder_embeding} is a local property, it suffices to consider bounded intervals only. Towards this end, write shortly $\Phi$ instead of $\Phi_\beta$ and pick $\lambda>\|t^{-\alpha}\omega_{\Phi,I}(f,t)\|_{L^\infty(0,\infty)}$. Then, by the Garsia, Rodemich, Rumsey lemma \cite[Lemma 1.1]{GRR} (see also \cite[Lemma 5.1]{OSK} for the infinite-dimensional version),
$$
\int_{I\times I}\Phi\left(\frac{\|f(a)-f(b)\|}{\lambda|a-b|^\alpha}\right)\,da\,db\le 2|I|,
$$
hence
$$
\|f(x)-f(y)\|\le 8\lambda\alpha\int_0^{|x-y|}u^{\alpha-1}\Phi_{-1}(2|I|u^{-2})\,du
$$
holds for all points of Lebesgue density $x,y\in I$.

As far as the inequality \eqref{mod_cb_embeding} with $I=\Bbb R$ is concerned, choosing $p\in (1, \infty)$ such that $\alpha-\frac1p>0$ we have (see \cite[Theorem 2.8.1(c)]{Tr1})
\[\|f\|_{L^\infty(I;X)}\le C_{\alpha,p}\|f\|_{B^\alpha_{p,q}(I;X)}\leq C_{\alpha,\beta,p} \|f\|_{B^\alpha_{\Phi_\beta,q}(I;X)},\]
where the latter estimate follows from $C_{\gamma} x^\gamma\leq e^x-1$ for $x\geq 0$. For other $I$, one uses an extension argument based on \eqref{eq:ext_oper}.
\end{proof}

\subsection{Extensions by zero}\label{subsec:ext_stop}

Below, $X$ is a Banach space, $p\in(1,\infty]$, $\alpha\in(0,1)$, $\beta\in[1,\infty)$ and $\Phi_\beta(x)=\exp(x^\beta)-1$. If $\alpha p>1$, then there exists a constant $C$ such that
\begin{align}
\|f\|_{B^\alpha_{p,\infty}(\Bbb R;X)}&\le C\|f\|_{B^\alpha_{p,\infty}(0,\infty;X)}\label{ext_by_zero_p}
\\
\|f\|_{B^\alpha_{\Phi_\beta,\infty}(\Bbb R;X)}&\le C\|f\|_{B^\alpha_{\Phi_\beta,\infty}(0,\infty;X)}\label{ext_by_zero_exp}
\end{align}
hold for every continuous function $f:\Bbb R\to X$ such that $f=0$ on $(-\infty,0]$.

\begin{proof}
Let $\mathcal N$ denote either $x^p$ or $\Phi_\beta$. Then, $\omega_{\mathcal N,\Bbb R}(f,t)\le\omega_{\mathcal N,\Bbb R_+}(f,t)+\|f\|_{L^{\mathcal N}(0,t;X)}$,  $\|f\|_{L^{\mathcal N}(0,t;X)}\le\|f\|_{B^\alpha_{\mathcal N,\infty}(0,t;X)}$ and, for small $t>0$,
\begin{align*}
\|f\|_{L^{\mathcal N}(0,t;X)}&\le\|f\|_{B^\alpha_{\mathcal N,\infty}(\Bbb R_+;X)}\|\zeta\|_{L^{\mathcal N}(0,t;X)}
\\
&\le\|f\|_{B^\alpha_{\mathcal N,\infty}(\Bbb R_+;X)}\zeta(t)\|1\|_{L^{\mathcal N}(0,t;X)}
\\
&\le\|f\|_{B^\alpha_{\mathcal N,\infty}(\Bbb R_+;X)}\zeta(t)/\mathcal N^{-1}(t^{-1})
\\
&\le Ct^\alpha\|f\|_{B^\alpha_{\mathcal N,\infty}(\Bbb R_+;X)}
\end{align*}
where $\zeta(x)=cx^{\alpha-\frac 1p}$ or $\zeta(x)=x^\alpha|\log x|^\frac 1\beta$ for small $x>0$ respectively by \eqref{holder_embeding} and \eqref{mod_holder_embeding}.
\end{proof}

For $\alpha\in (0,\tfrac1p)$ a more general result holds (see \cite{Runst-Sickel96} a full treatment of the subject).

\begin{lemma}\label{lem:pointwise}
Let $p\in [1, \infty)$, $q\in [1, \infty]$ and $\alpha\in (0,\frac1p)$. There exists a constant $C>0$ such that
\begin{equation}\label{eq:extensionzero2}
\|\one_{(0,\infty)} f\|_{B^{\alpha}_{p,q}(\R;X)}\leq C \|f\|_{B^{\alpha}_{p,q}(\R;X)}
\end{equation}
for every $f\in B^{\alpha}_{p,q}(\R;X)$
\end{lemma}
\begin{proof}
For convenience of the reader we give a self-contained argument here.
By real interpolation and reiteration (see \cite[Section 1.10]{Tr1}) it suffices to consider $q=p$. In that case $B^{s}_{p,p}(\R;X) = W^{s,p}(\R;X)$ has an equivalent norm given by $\|f\|_{L^p(\R)} + [f]_{W^s_{p}(\R^d,w;X)}$, where
\[[f]_{W^s_{p}(\R;X)}^p = \int_{\R} \int_{\R} \frac{\|f(x) - f(y)\|^p}{|x-y|^{sp+1}} \ud x\ud y = 2 \int_{0}^\infty \int_{\R} \frac{\|f(x+h) - f(x)\|^p}{|h|^{sp+1}} \ud x\ud h  ,\]
Now to prove the result let $f \in B^{s}_{p,p}(\R;X)$, and write $g = \one_{(0,\infty)} f$. By an elementary calculation one sees that
\[[g]_{W^s_{p}(\R;X)}\leq [f]_{W^s_{p}(\R;X)} + 4 \Big(\int_{\R} |x|^{-sp} \|g(x)\|^p \ud x\Big)^{\frac1p}.\]
The second term can be bounded by $C [f]_{W^s_{p}(\R;X)}$ using the fractional Hardy inequality (see \cite[Theorem 2b]{KMP}).
\end{proof}

\section{Temporal regularity of stochastic integrals}\label{sec:Stochint}

A Banach space $(X,\|\cdot\|)$ is called $2$-smooth  if there exists a constant $C>0$ such that
\[\|x+y\|^2+ \|x-y\|^2\leq 2\|x\|^2+2C\|y\|^2, \ \ x,y\in X.\]
Hilbert spaces are $2$-smooth, but also $L^p$, Sobolev spaces $W^{s,p}$, Besov spaces $B^{\alpha}_{p,q}$ and Triebel-Lizorkin spaces $F^{\alpha}_{p,q}$ for $p,q\in [2,\infty)$, $s>0$ and $\alpha\in\Bbb R$. A detailed study of $2$-smooth Banach spaces (and more general properties) can be found in \cite{Pisbook}.  In particular, it is shown there that a Banach space has the so-called martingale type $2$ property if and only if (up to an equivalent norm) $X$ is $2$-smooth. This class of Banach spaces allows for a variant of the stochastic integration theory similar to the scalar case (see \cite{B_1997,Ondr04}). For further details on stochastic integration in Banach spaces we refer the reader to the survey \cite{NVW15}.

Let $(X,\|\cdot\|)$ be a separable $2$-smooth Banach space and $H$ a separable Hilbert space. Assume $(\Omega, \A, \P)$ is a probability space with filtration $\F:=(\F_t)_{t\geq 0}$ such that $\F_0$ contains all $\P$-negligible sets from $\A$. Let $\F_{+} = (\F_{t+})_{t\geq 0}$. Let $\mathcal{P}$ and $\mathcal{P}_+$ denote the progressive $\sigma$-algebra with respect to $\F$ and $\F_+$ respectively.
Let $W$ be an $H$-cylindrical Brownian motion. For $p\in [0,\infty]$, $q\in [1, \infty]$ and $T\ge 0$ let $L^{p}_{\F}(\Omega;L^q(0,T;X)$ be the closure of the adapted strongly measurable processes in $L^{p}(\Omega;L^q(0,T;X)$. Recall from \cite[Theorem 1]{OS} that such processes have a progressive measurable modification. Let $\gamma(H,X)$ denote the space of {\em $\gamma$-radonifying operators} from $H$ into $X$ (see e.g. \cite{HNVW2} for a definition).

Let $W$ be an $H$-cylindrical Wiener process. Due to the geometric condition on $X$ for $f\in L^{0}_{\F}(\Omega;L^2(0,T;\gamma(H,X)))$ we can define the indefinite stochastic integral by $f\cdot W \in L^0(\Omega;C([0,T];X))$ by
\[f\cdot W(t) = \int_0^t f(s) d W(s), \ \ t\in [0,T].\]
The Burkholder-Davis-Gundy inequality obtained in \cite{Seid10} implies that there exists a constant $K$ depending on $X$ such that, for all $p\in [1, \infty)$, $T\ge 0$ and for all adapted $f\in L^p(\Omega;L^2(0,1;\gamma(H,X)))$,
\begin{equation}\label{eq:BDG}
\Big( \E\sup_{t\in [0,T]} \Big\|f\cdot W\Big\|^p\Big)^{1/p} \leq K \sqrt{p} \|f\|_{L^p(\Omega;L^2(0,T;\gamma(H,X)))}.
\end{equation}
It is much simpler to check the same result with a different dependence on $p$. However, the factor $\sqrt{p}$ is essential in the proofs below. The growth rate $\sqrt{p}$ is optimal already in the scalar case. This follows for instance by taking $f = \one_{[0,T]}$.

\begin{lemma}\label{lem:bmina}
Let $X$ be a separable uniform $2$-smooth Banach space. Let $T>0$, $p\in [1, \infty)$ and $q\in (2, \infty]$.
Let $f\in L_{\mathcal F}^p(\Omega;L^q(0,T;\gamma(H,X)))$ and let $M(t) = f\cdot W$. Then for all $0\le a\le t\le T$,
\begin{align*}
\Big(\E(\|M_t-M_{a}\|^{p}|\F_{a})\Big)^{1/p}\leq K {p}^{1/2} \|f\|_{L^p(\Omega;L^q(0,T;\gamma(H,X)))} (t-a)^{\frac12-\frac1q} \ \text{a.s.}
\end{align*}
\end{lemma}
\begin{proof}
Let $F\in\mathcal F_a$. Then by \eqref{eq:BDG} we have
\begin{align*}
\Bbb E\big(\mathbf 1_F\|M_t-M_a\|^{p}\big)
& =\Bbb E\,\Big\|\int_a^t\mathbf 1_Ff\ud W\Big\|^{p}
\\ & \le K^{p}p^{p/2}\Bbb E\,\Big|\int_a^t\mathbf 1_F\|f\|^2_{\gamma(H,X)}\ud s\Big|^{p/2}
\\ &
\le\Bbb E\,\big(\mathbf 1_F K^{p}{p}^{p/2}\|f\|^{p}_{L^q(0,1;\gamma(H,X))}(t-a)^{\frac{p}{2} - \frac{p}{q}}\big),
\end{align*}
where on the last line we applied H\"older's inequality. Hence
$$
\Bbb E\,\big(\|M_t-M_a\|^{p}|\mathcal F_a\big)\le K^{p}{p}^{p/2}\|f\|^{p}_{L^p(\Omega;L^q(0,1;\gamma(H,X)))}(t-a)^{\frac{p}{2} - \frac{p}{q}}\quad\text{a.s.}
$$
\end{proof}

The following is our main result on the regularity of the indefinite stochastic integral. It provides the optimal path regularity properties and  norm estimates.

\begin{theorem}\label{thm:mainindef}
Let $(X,\|\cdot\|)$ be a separable $2$-smooth Banach space. Then there exists an increasing positive function $(C_t)_{t\ge 0}$ such that
\begin{enumerate}[(i)]
\item\label{it:regpas} $f\cdot W\in B^{\alpha}_{\Phi_2,\infty}(0,T;X)$ a.s.,
\item\label{it:regp} $(\E\|f\cdot W\|_{B^{\alpha}_{p,\infty}(0,T;X)}^{2p})^{1/(2p)} \leq C_T p^{1/2} \|f\|_{L^{2p}(\Omega;L^q(0,T;\gamma(H,X)))}$,
\item\label{it:regPhi} $(\E\|f\cdot W\|_{B^{\alpha}_{\Phi_2,\infty}(0,T;X)}^p)^{1/p}\leq C_T p^{1/2} \|f\|_{L^{\infty}(\Omega;L^q(0,T;\gamma(H,X)))}$,

\item\label{it:reg_L_Phi} $\|f\cdot W\|_{L^{\Phi_2}(\Omega;B^{\alpha}_{\Phi_2,\infty}(0,T;X))}\leq C_T\|f\|_{L^{\infty}(\Omega;L^q(0,T;\gamma(H,X)))}$
\item\label{it:reginftyprob} $\P(\|f\cdot W\|_{B^{\alpha}_{\Phi_2,\infty}(0,T;X)}>\varepsilon,\,\|f\|_{L^q(0,T;\gamma(H,X))}\le\delta)\leq 2\exp\,\{-C_T^{-2}\delta^{-2}\varepsilon^2\}$
\item\label{it:finalest} $\|f\cdot W\|_{L^p(\Omega;B^{\alpha}_{\Phi_2,\infty}(0,T;X))}\leq C_{T} p^{1/2}\|f\|_{L^{N_p}(\Omega;L^q(0,T;\gamma(H,X)))}$.
\end{enumerate}
hold for all $T,\varepsilon, \delta\in(0,\infty)$, $p\in [1, \infty)$, $q\in (2, \infty]$ and $f\in L^{0}_{\F}(\Omega;L^q(0,T;\gamma(H,X)))$ where $\alpha = \frac12-\frac1q$ and $N_p(t) = t^{p} \log^{p/2}(t+1)$.
\end{theorem}

Part of the argument is inspired by the dyadic norm equivalence \eqref{eq_dyadic} which was used in \cite[Theorem 4.1]{HytVer} and \cite{Vercor} for Gaussian processes.
\begin{proof}
Let us start with the case $T=1$ and write $E_q = L^q(0,1;\gamma(H,X))$. To prove \eqref{it:regp}, assume that $f\in L^{2p}(\Omega;E_q)$ and denote $M = f\cdot W$ and
\[Y_{n,p}:=2^{n \alpha}\|M(\cdot+2^{-n})-M\|_{L^p(I(2^{-n});X)}.\]
We may write
\begin{align*}
Y^p_{n,p} &= \int\limits_0^{1-2^{-n}} 2^{np \alpha}\|M(t+2^{-n}) -
M(t)\|^p  \ud t
\\ & = \sum_{m=1}^{2^n-1} \int\limits_{(m-1) 2^{-n}}^{m 2^{-n}} 2^{np\alpha}\|M(t+2^{-n}) - M(t)\|^p  \ud t
\\ & = \sum_{m=1}^{2^n-1} 2^{-n}\int\limits_{0}^{1} 2^{np\alpha}\|M((s+m)2^{-n}) - M((s+m-1)2^{-n})\|^p
\ud s
\\ & =\int\limits_{0}^{1}  2^{-n}\sum_{m=1}^{2^n-1} \eta_{n,m,s} \ud s
\end{align*}
Here \(\eta_{n,m,s} = 2^{n \alpha p} \|M((s+m)2^{-n}) - M((s+m-1)2^{-n})\|^p\).
Letting
\[Z_{n,p}^p =  \int\limits_{0}^{1} 2^{-n} \sum_{m=1}^{2^n-1} \xi_{n,m,s}  \ud s \ \ \ \text{and} \ \  \ \xi_{n,m,s} = \E(\eta_{n,m,s}|\F_{(s+m-1)2^{-n}})\]
it follows that $s\mapsto\xi_{n,m,s}$ are non-negative, progressively measurable processes on $[0,1]$ (see e.g. \cite[Corollary 0.2]{OS}), $\eta_{n,m,s}$ and $\xi_{n,m,s}$ are uniformly bounded in $L^2(\Omega)$ with respect to $s\in[0,1]$ for every $n\ge 1$ and $1\le m<2^n$ and, for fixed $n\geq 1$ and $s\in[0,1]$, $(\eta_{n,m,s} -\xi_{n,m,s})_{m=1}^{2^n-1}$ is a sequence of orthogonal random variables in $L^2(\Omega)$. If we take second moments we may use Jensen's inequality to obtain
\begin{align*}
\E |Y_{n,p}^p - Z_{n,p}^p|^2 &=
\E \Big|\int\limits_{0}^{1}
\big[2^{-n} \sum_{m=1}^{2^n-1} \eta_{n,m,s} - \xi_{n,m,s}\big]
\ud s\Big|^2
\\ & \leq \int\limits_{0}^{1}  \E \Big| 2^{-n} \sum_{m=1}^{2^n-1}\big(\eta_{n,m,s} -\xi_{n,m,s}\big)\Big|^2 \ud s
\\ & = \int\limits_{0}^{1}  2^{-2n}  \sum_{m=1}^{2^n-1}\E \big| \eta_{n,m,s} -\xi_{n,m,s}\big|^2 \ud s
\\ & \leq \int\limits_{0}^{1}  2^{-2n}  \sum_{m=1}^{2^n-1} \E \eta_{n,m,s}^2 \ud s,
\end{align*}
where on the last line we used
$$
\E \big| \eta_{n,m,s} -\xi_{n,m,s}\big|^2 = \E \eta_{n,m,s}^2 - \E \xi_{n,m,s}^2\leq  \E \eta_{n,m,s}^2
$$
which follows from properties of the conditional expectation.

By \eqref{eq:BDG} and H\"older's inequality, we have
\begin{align*}
\E \eta_{n,m,s}^{2} & \leq K^{2p}_X p^{p} 2^{2\alpha np}\E \Big(\int_{(s+m-1)2^{-n}}^{(s+m)2^{-n}} \|f(r)\|_{\gamma(H,X)}^2\ud r\Big)^{p}
\\ & \leq K_X^{2p} p^{p}  \|f\|_{L^{2p}(\Omega;E_q)}^{2p}.
\end{align*}
It follows that
\begin{align}\label{eq:help1}
\E \sum_{n\geq 1} | Y^p_{n,p} - Z_{n,p}^p|^2 & \leq K_X^{2p} p^{p}  \|f\|_{L^{2p}(\Omega;E_q)}^{2p}
\end{align}
which implies
\[\lim_{n\to \infty}(Y^p_{n,p} - Z_{n,p}^p) = 0 \ \ \text{a.s.}\]
In order to show that $Y_{n,p}^p$ is bounded a.s.\ we will prove that
$Z_{n,p}^p$ is uniformly bounded a.s. Indeed, by Lemma \ref{lem:bmina} and the Fubini theorem, we have a.s.
\begin{align}\label{eq:help2}
Z_{n,p}^p & \leq \sup_{1\leq m\leq 2^n-1}\xi_{n,m,s}
\leq  K^{p}{p}^{p/2}\|f\|^{p}_{L^p(\Omega;E_q)}.
\end{align}

Therefore, from \eqref{eq_dyadic} we can conclude that $M\in B^{\alpha}_{p,\infty}(0,1;X)$ a.s.\ with
\begin{align}\label{eq:helpMY}
\|M\|_{p,I,\infty,\alpha}^p = \sup_{n\geq 1} Y_{n,p}^p \leq \Big(\sum_{n\geq 1} |Y_{n,p}^p - Z_{n,p}^p|^2\Big)^{1/2} + \sup_{n\geq 1} Z_{n,p}^p
\end{align}
and taking $L^2(\Omega)$-norms and applying the triangle inequality yields
$$
(\E \|M\|_{p,I,\infty, \alpha}^{2p})^{1/2} \leq \Big(\E\sum_{n\geq 1} |Y_{n,p}^p - Z_{n,p}^p|^2\Big)^{1/2} + (\E \sup_{n\geq 1} Z_{n,p}^{2p})^{1/2}\leq c_X^{p} p^{p/2}  \|f\|_{L^{2p}(\Omega;E_q)}^{p},
$$
where the last estimate follows from \eqref{eq:help1} and \eqref{eq:help2} and $c_X$ is a constant depending only on $X$. Similarly, by \eqref{eq:BDG}, one has that
\[\E\|M\|_{L^p(0,1;X)}^p\leq K^p p^{p/2} \|f\|_{L^{p}(\Omega;E_q)}^p\]
holds. Combining the estimates we get \eqref{it:regp} by Proposition \ref{prop_16}.

\eqref{it:regPhi}:  Assume that $f\in L^\infty_{\F}(\Omega;E_q)$. We use the equivalent norm given in \eqref{eq_dyadic}. Then, using the equivalence \eqref{eq:lux_norm} and $(1+\Phi_2(x))^p = 1+\Phi_2(\sqrt p x)$, we get
\begin{align*}
\E\|f\cdot W\|_{\Phi_2,I,\infty, \alpha}^p &\le\E \sup_{n\geq 1} \inf_{\lambda>0} \frac{1}{\lambda^p} \Big(1+  \sum_{k\geq 1} \frac{p^{k}\lambda^{2k}}{k!} Y^{2k}_{n,2k} \Big)
\\ & \leq \inf_{\lambda>0} \frac{1}{\lambda^p} \Big(1+  \sum_{k\geq 1} \frac{p^{k}\lambda^{2k}}{k!}  \E \sup_{n\geq 1} Y^{2k}_{n,2k} \Big)
\end{align*}
Now by Jensen's inequality and \eqref{eq:helpMY} we can write
\begin{align*}
\E \sup_{n\geq 1} Y^{2k}_{n,2k} \leq \Big(\E \sup_{n\geq 1} Y^{4k}_{n,2k}\Big)^{1/2}
\leq c_X^{2k} \|f\|_{L^\infty(\Omega;E_q)}^{2k} k^{k}.
\end{align*}
Therefore, using  $k^k/k!\leq e^k$ we find that
\begin{align*}
\E\|f\cdot W\|_{\Phi_2, \infty, \alpha}^p & \le \inf_{\lambda>0} \frac{1}{\lambda^p} \Big(1+  \sum_{k\geq 1} p^{k}\lambda^{2k} e^k c_X^{2k} \|f\|_{L^\infty(\Omega;E_q)}^{2k}\Big)
\\ & \leq {\frac{2^{p+2}}{3}}p^{p/2} e^{p/2}c_X^p\|f\|_{L^\infty(\Omega;E_q)}^p,
\end{align*}
by setting $\lambda^{-1} =2p^{1/2}e^{1/2}c_X\|f\|_{L^\infty(\Omega;E_q)}$. Similarly, one shows by \eqref{eq:BDG} that
\[\E\|f\cdot W\|_{L^{\Phi_2}(0,1;X)}^p\le\kappa_X^p p^{p/2} \|f\|_{L^\infty(\Omega;E_q)}^p\]
and therefore, the required estimate follows.

\eqref{it:reg_L_Phi}: follows directly from \eqref{it:regPhi} and a standard power series argument \cite[Theorem 3.4]{CI_1993}.

\eqref{it:regpas} and \eqref{it:reginftyprob}: Assume $f\in L^0_{\F}(\Omega;E_q)$ and fix a $\mathcal{P}$-measurable version of $f$.
Observe that $t\mapsto \|f\one_{[0,t]}\|_{E_q}$ is an increasing adapted process. For $q<\infty$ this is clear from the continuity, and for $q=\infty$, this follows from the equality
\[\|f\one_{[0,t]}\|_{E_\infty} = \sup_{q\in \N\setminus\{0,1\}} \|f\one_{[0,t]}\|_{E_q}.\]
Now define
\[\tau_s = \inf\{t\in [0,1]: \|f\one_{[0,t]}\|_{E_q}>s\},\qquad s>0,\]
where we take $\tau_s = \infty$ if the infimum is taken over the empty set. Then $\tau_s$ is an $\F_+$-stopping time as $t\mapsto \|f\one_{[0,t]}\|_{E_q}$ is increasing and adapted.

It follows that $f^{(s)} := f \one_{[0,\tau_s)}$ is $\mathcal{P}_+$-measurable, $\|f^{(s)}\|_{E_q}\leq s$ by definition of $\tau_s$ and $f^{(s)}\to f$ in $L^0(\Omega;E_q)$ as $s\uparrow\infty$.
Let $M^{(s)}:=f^{(s)}\cdot W$. Then $M^{(s)}\in B^{\alpha}_{\Phi_2,\infty}(0,1;X)$ a.s. Moreover, $M^{(s)}(t) = M(t\wedge \tau_s)$ and therefore, letting $s\uparrow\infty$ we find $M\in B^{\alpha}_{\Phi_2,\infty}(0,1;X)$ a.s. Now the tail estimate in \eqref{it:reginftyprob} follows from \eqref{it:reg_L_Phi} and the Chebychev inequality since, defining $\lambda=C^{-1}$,
\begin{align*}
\Bbb P\,(\|M\|_{B^\alpha_{\Phi_2,\infty}(0,1;X)}>\varepsilon,\,\|f\|_{E_q}\le C^{-1})&=\Bbb P\,(\|M\|_{B^\alpha_{\Phi_2,\infty}(0,1;X)}>\varepsilon,\,\tau_\lambda=\infty)
\\
&=\Bbb P\,(\|M^{(\lambda)}\|_{B^\alpha_{\Phi_2,\infty}(0,1;X)}>\varepsilon,\,\tau_\lambda=\infty)
\\
&\le\Bbb P\,(\|M^{(\lambda)}\|_{B^\alpha_{\Phi_2,\infty}(0,1;X)}>\varepsilon)
\\
&\le\min\,\left\{1,[\Phi_2(\varepsilon)]^{-1}\right\}\le 2e^{-\varepsilon^2}.
\end{align*}
The general inequality follows from applying the above inequality to $\widetilde f:=(C\delta)^{-1}f$ and the corresponding $\widetilde M$ and taking $\varepsilon$ appropriately.

The final assertion \eqref{it:finalest} follows from \eqref{it:reginftyprob} and Lemma \ref{lem:ax_gauss}. Note that the constant $10$ can be absorbed into the constant $C_T$.

If $T$ is general then define $\mathcal G_t=\mathcal F_{Tt}$, $f_T(t)=f(Tt)$ and $W_T(t)=T^{-1/2}W(Tt)$. Then $W_T$ is a cylindrical $(\mathcal G_t)_{t\ge 0}$-Wiener process and  $$T^{1/2}(f_T\cdot W_T)(t)=(f\cdot W)(Tt),\qquad t\ge 0$$ holds by linear substitution. We apply \eqref{it:regp}-\eqref{it:reginftyprob} to $T^{1/2}f_T$ and $W_T$ on $(0,1)$ and we obtain the general case on $(0,T)$ by scaling \eqref{eq:BO_scaling}. Finally, we realize that if \eqref{it:regp}-\eqref{it:reginftyprob} hold on $(0,T)$ with some constant $C_T$ then \eqref{it:regp}-\eqref{it:reginftyprob} hold on $(0,\tau)$ for every $0<\tau<T$ with the same constant $C_T$. Therefore $T\mapsto C(T)$ can be constructed as an increasing function.
\end{proof}

\begin{lemma}\label{lem:ax_gauss} Let $\kappa$ be a positive constant, $p\in[1,\infty)$ and let $N(t) = t^{p} \log^{p/2}(t+1)$. Then there exists a positive constant $c$ such that, whenever $A$ and $B$ are non-negative random variables and
$$
\Bbb P\,(A\ge x,\,B\le y)\le 2\exp\,\{-\kappa y^{-2}x^2\},\qquad x,y\in(0,\infty)
$$
then $\|A\|_{L^p(\Omega)}\le  10 p^{1/2} \kappa^{-1/2}\|B\|_{L^N(\Omega)}$.
\end{lemma}

\begin{proof}
Define $\mu=2^{-1}p^{-1/2}\kappa^{1/2}$. By homogeneity we can assume that $\|B\|_{L^N(\Omega)} = \mu$.
Let $N^{-1}$ denote the inverse of the function $N$. Then
$(N^{-1}(x))^p\log^{p/2}(N^{-1}(x)+1) = x$  for $x\geq0$, and $x^{1/(2p)} \leq N^{-1}(x)$ for $x\geq 1$. Therefore,
\begin{align*}
\Bbb E\,A^p&=\Bbb E\,N(B/\mu)+ p\int_0^\infty x^{p-1}\Bbb P\,(A>x,\,(N(B/\mu))^{1/p}\le x)\,dx
\\
&\le 1+ 2p\int_0^\infty x^{p-1}\exp\,\{-\kappa x^2\mu^{-2}(N^{-1}(x^p))^{-2}\}\,dx
\\
&=1+2\int_0^\infty\exp\,\{-\kappa x^{2/p}\mu^{-2}(N^{-1}(x))^{-2}\}\,dx
\\
&= 1+2\int_{0}^\infty [N^{-1}(x)+1]^{-\frac{\kappa}{\mu^2}}\,dx
\\
&\le 3+2\int_{1}^\infty [N^{-1}(x)+1]^{-\frac{\kappa}{\mu^2}}\,dx
\\
&\le 3+2\int_{1}^\infty x^{-\frac{\kappa}{2p\mu^2}}\,dx = 5,
\end{align*}
where we used the definition of $\mu$. Therefore, $\|A\|_{L^p(\Omega)}\leq 5^{1/p} 2 (p/\kappa)^{1/2} \|B\|_{L^N(\Omega)}$.
\end{proof}

\begin{remark}\label{rem:othermoments} \
By \cite[Proposition IV.4.7]{RY} the $L^{2p}$-estimate in Theorem \ref{thm:mainindef} \eqref{it:regp} can be extrapolated to all $r\in (0,2p]$.
\end{remark}

\section{Temporal regularity of deterministic convolutions}

Let $X$ be a Banach space. Let $A$ be the generator of an analytic $C_0$-semigroup $(S(t))_{t\geq 0}$. We write $R(\lambda,A) = (\lambda-A)^{-1}$ for the resolvent of $A$ for $\lambda\in \rho(A)$, where $\rho(A)$ denotes the resolvent set of $A$. We say that a $C_0$-semigroup $(S(t))_{t\geq 0}$ with generator $A$ is exponentially stable if there exist $M,\omega>0$ such that $\|S(t)\|\leq M e^{-\omega t}$. We will always set $S(t) = 0$ for $t<0$. Note that for an exponentially stable analytic semigroup, one has
\[\{\lambda\in \C:\Re(\lambda)\leq 0\}\subseteq \rho(A) \ \ \text{and} \ \ \sup_{s\in \R}\|sR(is,A)\|<\infty,\]
and the Fourier transform $\F$ of $S$ satisfies $\F S(s) = -R(is, A)$. For details on semigroup theory we refer the reader to \cite{EN}.

Below we discuss a maximal regularity result in the scale of Besov-Orlicz functions.
Previous regularity and Fourier multiplier results for evolution equations on Besov spaces can be found in \cite{Am97}. Below we discuss a result on general Besov-Orlicz spaces.

Let $\mathcal N$ be a Young function (see section \ref{subsec:Orlicz}) and $g\in L^{\mathcal N}(\R;X)$. Then the convolution $S*g$ is well-defined a.e.\ by Lemma \ref{young}.

\begin{lemma}\label{young}
Let $f\in L^{\mathcal N}(\Bbb R^d;\calL(X))$ and $g\in L^1(\Bbb R^d;X)$. Then
\[\|f*g\|_{L^{\mathcal N}(\Bbb R^d;X)}\le\|f\|_{L^{\mathcal N}(\Bbb R^d;\calL(X))}\|g\|_{L^1(\Bbb R^d;X)}.\]
\end{lemma}

\begin{proof} Define $\mu=\|g\|_{L^1(\Bbb R^d)}$ and a probability measure $d\theta=\mu^{-1}|g|\ud x$. Then
$$
\mathcal N(\|f*g(x)\|_X/\lambda)\le\int_{\Bbb R^d}\mathcal N(\mu\|f(x-y)\|/\lambda)\ud \theta(y)
$$
by the Jensen inequality. If $\lambda>\mu\|f\|_{L^{\mathcal N}(\Bbb R^d;\calL(X))}$ then, integrating both sides, we get
$$
\int_{\Bbb R^d}\mathcal N(\|f*g(x)\|_X/\lambda)\ud x\le 1,
$$
by the definition of the Luxemburg norm.
\end{proof}

The next result is formulated for $\alpha>0$, so that the convolution is well-defined by the above discussion. Using the theory of vector-valued tempered distributions and suitable approximation argument one can extend the result to any $\alpha\in \R$.
\begin{proposition}\label{prop:detconv2}
Let $\alpha>0$, let $\mathcal{N}$ be a Young's function and let $q\in [1, \infty]$. Assume that $A$ generates an analytic $C_0$-semigroup $(S(t))_{t\geq 0}$ which is exponentially stable.  Then there exists a constant $C$ depending only on $S$ such that
\begin{itemize}
\item $S*f\in\operatorname{Dom}\,(A)$ a.e.\ and
\item $|S*f|_{B_{\mathcal N,q}^\alpha(\R;D(A))}\leq C |f|_{B_{\mathcal N,q}^\alpha(\R;X)}$
\end{itemize}
holds for every $f\in B_{\mathcal N,q}^\alpha(\R;X)$.
\end{proposition}

\begin{proof}
We use the strategy of proof given in \cite[Theorem 6.1.6]{BeLo}. First consider the case $f\in B_{\mathcal N,q}^\alpha(\R;D(A))$. Then the integral $S*f$ is well-defined by Lemma \ref{young} and $S*f\in D(A)$ a.e. To prove the required estimate, since $A$ is invertible it is enough to estimate the norm of $A S*f$. Notice that for all $k\in \mathbb{N}_0$ we have
\[\varphi_k* f = \sum_{l=-1}^1 \psi_{k+l} * \varphi_k*f,\]
where $\psi_{k} = \varphi_k$ for all $k\neq 0$, $\psi_0=\varphi_0$ and $\psi_{-1}=0$. Fix $k\in
\mathbb{N}$, and denote $f_k=\varphi_k*f$. Then $f_k\in L^{\mathcal N}(\R;X)$ by Lemma \ref{young}.  For $n=-1, 0, 1, 2, \ldots$ we may
write
\[\psi_{n} *\varphi_k* S*f = \psi_{n}* S*f_k.\]
We estimate $\|{\varphi}_n*A S*f_k\|_{L^{\mathcal N}(\R;X)}$. For each $n\geq -1$,
we use Lemma \ref{young} to estimate
\[\begin{aligned}
\|{\psi}_{n}*AS*f_k\|_{L^{\mathcal N}(\R;X)} &=
 \|\mathcal{F}^{-1}(\widehat\psi_n \, AR(\cdot i,A)) f_k\|_{L^{\mathcal N}(\R;X)}
\\ & \leq \|\mathcal{F}^{-1}(\widehat\psi_n \,  AR(\cdot i,A))\|_{L^1(\R;\calL(X))}
\|f_k\|_{L^{\mathcal N}(\R;X)}.
\end{aligned}\]
Let $t>0$ be fixed. First consider $n\geq 1$. Clearly, it holds that
\[\begin{aligned}
\int_{r>t} & \|\mathcal{F}^{-1}(\widehat\psi_n \, AR(\cdot
i,A))(r)\|_{\calL(X)} \ud r \\ &= \int_{r>t} r^{-2}
\|\mathcal{F}^{-1}D^2( \widehat\psi_n \, AR(\cdot
i,A))(r)\|_{\calL(X)} \ud r
\\ &\leq \sup_{r\in \R}\|\mathcal{F}^{-1}D^2 (\widehat\psi_n \,
AR(\cdot i,A))(r)\|_{\calL(X)} \int_{r>t} r^{-2} \ud r
\\ & \leq \|(D^2 \widehat\psi_n \, AR(\cdot i,A))\|_{L^1(\R;\calL(X)}
\frac{1}{t},
\end{aligned}\]
where $D$ stands derivation. One also has that
\[\begin{aligned}
\int_{0\leq r<t} \|\mathcal{F}^{-1}(\widehat\psi_n \, AR(\cdot
i,A))(r)\|_{\calL(X)}\ud r &\leq t \|\widehat\psi_n \, A
R(\cdot i,A)\|_{L^1(\R;\calL(X)}.
\end{aligned}
\]
Therefore, we deduce that
\[\begin{aligned}
\|\mathcal{F}^{-1}&(\widehat\psi_n \, AR(\cdot i,A))\|_{L^1(\R;\calL(X))} \\ & \leq
t^{-1} \|D^2 (\widehat\psi_n \,AR(\cdot
i,A))\|_{L^1(\R;\calL(X))} + t \|\widehat\psi_n \, A
R(\cdot i,A)\|_{L^1(\R;\calL(X)}.
\end{aligned}\]
Minimization over $t>0$ gives
\begin{equation}\label{eq:FinverseL1}
\begin{aligned}
\|\mathcal{F}^{-1}(&\widehat\psi_n \, A R(\cdot
i,A))\|_{L^1(\R;\calL(X))}
\\ & \leq \|(D^2 \widehat\psi_n \, AR(\cdot
i,A))\|_{L^1(\R;\calL(X))}^{\frac12} \|\widehat\psi_n \,
AR(\cdot i,A)\|_{L^1(\R;\calL(X))}^{\frac12}.
\end{aligned}
\end{equation}

Since $\widehat\psi_n$ has support in $I_n := [-2^{n+1},-2^{n-1}] \cup
[-2^{n-1},-2^{n+1}]$ it follows that
\[\begin{aligned}
\|&D^2 (\widehat\psi_n \, AR(\cdot
i,A))\|_{L^1(\R;\calL(X))}
\\ & \leq \|D^2{\widehat\psi}_n\|_{\infty} \|AR(\cdot
i,A)\|_{L^1(I_n;\calL(X))} \\ & \ \ +
2\|D{\widehat\psi}_n\|_{\infty} \|A R(\cdot
i,A)^2\|_{L^1(I_n;\calL(X))} + \|\widehat\psi_n\|_{\infty}
\|AR(\cdot i,A)^3\|_{L^1(I_n;\calL(X))}
\\ & \leq
C_1 2^{-2n} 2^{n+1} + C_{2} 2^{-n}
+ C_{3} 2^{-(n+1)} \leq C_{4}
2^{-n}
\end{aligned}\]
where we used
\begin{equation}\label{eq:Rthetadelta}
\|A R(is,A)\|_{\calL(X)} + (1+|s|)\|R(is,A)\|_{\calL(X)} \leq C.
\end{equation}
Similarly one has that
\[\|\widehat\psi_n \, AR(\cdot i,A)\|_{L^1(\R;\calL(X))}
\leq C_{S,\psi} 2^{n}.
\]

Combining these estimates with \eqref{eq:FinverseL1} we arrive at
\[\|\mathcal{F}^{-1}(\widehat\psi_n \,
AR(\cdot i,A))\|_{L^1(\R;\calL(X))} \leq
C_{S,\psi}.
\]
The same type of estimates holds for $n=0$. We may conclude that
\[\begin{aligned}
|AS*f|_{B_{\mathcal N,q}^\alpha(\R;X)} &\leq \sum_{l=-1}^1 (\sum_{k\geq 0} (2^{\alpha
k} \|{\psi}_{k+l}*AS*f_k\|_{L^{\mathcal N}(\R;X)})^q)^{\frac1q}
\\ &\leq C_{S,\psi} (\sum_{k\geq 0} (2^{\alpha k} \|f_k\|_{L^{\mathcal N}(\R;X)})^q)^{\frac1q}
\\
& = C_{S,\psi} |f|_{B_{\mathcal N,q}^\alpha(\R;X)},
\end{aligned}\]
and the required estimate follows.

Now for general $f$, if $q<\infty$, then the required estimate follows by a density argument using the standard fact that $nR(n,A)\to I$ in the strong operator topology (see \cite[Proposition 10.1.7]{HNVW2}). If $q=\infty$ we use a similar approximation argument, but a density does not work in general. Let $f\in B_{\mathcal N,\infty}^\alpha(\R;X)$ and consider $f^{(n)} = nR(n,A) f$ for $n\geq 1$. Then $f^{(n)}\in B_{\mathcal N,q}^\alpha(\R;D(A))$ and by the previous estimates applied to $f^{(n)}$, we have
\[|S*f^{(n)}|_{B_{\mathcal N,\infty}^\alpha(\R;D(A))} \leq C |f^{(n)}|_{B_{\mathcal N,\infty}^\alpha(\R;X)}\leq \tilde{C} |f|_{B_{\mathcal N,\infty}^\alpha(\R;X)}. \]
Since $f\in B_{\mathcal N,q}^{\beta}(\R;X)$ for any $\beta<\alpha$ and $q\in [1, \infty]$, it follows that $S*f^{(n)} \to S*f$ in $B_{\mathcal N,q}^\beta(\R;D(A))$ for any $q\in [1, \infty)$. In particular, for every $k\geq 0$, $\varphi_k*S*f^{(n)}\to \varphi_k*S*f$ in $L^{\mathcal{N}}(\R;D(A))$. Therefore, for every $k\geq 0$,
\[2^{\alpha k} |\varphi_k*A S*f|_{L^{\mathcal{N}}(\R;X)} = 2^{\alpha k} |\varphi_k*A S*f^{(n)}|_{L^{\mathcal{N}}(\R;X)} \leq \tilde{C} |f|_{B_{\mathcal N,\infty}^\alpha(\R;X)}. \]
Taking the supremum over all $k\geq 0$ the result for $q=\infty$ follows as well.
\end{proof}

\begin{remark}
Analogous results to those in Proposition \ref{prop:detconv2} do not hold with $B^{\alpha}_{\mathcal{N}, q}$ replaced by $L^{\mathcal{N}}$ in general (except if $X$ is a Hilbert space). We refer to \cite{KuWe} for a detailed discussion on maximal regularity on $L^p$-spaces. Most result extend to the setting of Orlicz spaces by standard extrapolation arguments for singular integrals.
\end{remark}

\begin{theorem}\label{cor:detconv2}
Let $p\in(1,\infty)$, $\alpha\in(0,1)$, $\alpha \neq 1/p$, $\beta\in[1,\infty)$, $T>0$ and let $(S(t))_{t\geq 0}$ be an analytic $C_0$-semigroup generated by $A$. Let $\mathcal N\in\{x\mapsto x^p,\Phi_\beta\}$. Then there exists a constant $C$ such that, for every $f\in B_{N,\infty}^\alpha(0,T;X)$, satisfying $f(0+)=0$ if $N(x) = x^p$ and $\alpha p>1$, the convolution integral
\[
u(t)=\int_0^tS(t-s)f(s)\ud s
\]
converges for a.e.\ $t\in[0,T]$, $u\in\operatorname{Dom}\,(A)$ a.e.\ in $[0,T]$ and
\[
\|Au\|_{B_{N,\infty}^\alpha(0,T;X)}\leq C \|f\|_{B_{N,\infty}^\alpha(0,T;X)}.
\]
\end{theorem}

\begin{proof}
Define $F=\mathcal E_{[0,T]}f$ on $[0,\infty)$ and $F=0$ on $(-\infty,0)$ where $\mathcal E_{[0,T]}$ is the extension operator from Proposition \ref{prop:ext}, and choose $\lambda\in\Bbb R$ such that $U(t)=e^{\lambda t}S(t)$ is exponentially stable. Then
\begin{align*}
\|(A+\lambda I)U*F\|_{B_{\mathcal N,\infty}^\alpha(\R;X)}&\leq C_0\|F\|_{B_{\mathcal N,\infty}^\alpha(\R;X)}
\\
&\leq C_1\|\mathcal E_{[0,T]}f\|_{B_{\mathcal N,\infty}^\alpha(\R_+;X)}
\\
&\leq C_2\|f\|_{B_{\mathcal N,\infty}^\alpha(0,T;X)}
\end{align*}
by Proposition \ref{prop:detconv2}, \eqref{ext_by_zero_p}, \eqref{ext_by_zero_exp}, \eqref{eq:extensionzero2}  and \eqref{eq:ext_oper}. Since
\[
\|U*F\|_{B_{\mathcal N,\infty}^\alpha(\R;X)}\leq\|(A+\lambda I)^{-1}\|\|(A+\lambda I)U*F\|_{B_{\mathcal N,\infty}^\alpha(\R;X)}
\]
and $U*f=U*F$ a.e.\ on $[0,T]$, we conclude that
\begin{equation}\label{eq:firest}
\|AU*f\|_{B_{\mathcal N,\infty}^\alpha(0,T;X)}\leq C_3\|f\|_{B_{\mathcal N,\infty}^\alpha(0,T;X)}.
\end{equation}
Now by real interpolation there exists $\kappa_T$ such that
\[\|ab\|_{B^\alpha_{\mathcal N,\infty}(0,T;X)}\le\kappa_T\|a\|_{C^{0,1}[0,T]}\|b\|_{B^\alpha_{\mathcal N,\infty}(0,T;X)}\]
so
\eqref{eq:firest} yields the result as $AS*f=e^{-\lambda\cdot}[AU*(e^{\lambda\cdot}f)]$ on $[0,T]$.
\end{proof}

\section{Temporal regularity of stochastic convolutions\label{sec:SPDE}}

Let $X$ be a separable $2$-smooth Banach space, let $A$ be the generator of an analytic $C_0$-semigroup $(S(t))_{t\geq 0}$ on $X$. If $f$ belongs to $L^0_{\F}(\Omega;L^2(0,T;\gamma(H,X)))$ then we define a stochastic convolution integral $S\diamond f$ by
\[
(S\diamond f)(t) := \int_0^t S(t-s) f(s) d W(s),\qquad t\in[0,T].
\]
Since $S\diamond f$ is continuous in probability, we can assume that $S\diamond f$ is predictable (see e.g. \cite[Proposition 3.2]{DPZ}.)

Next we prove our main result. Theorem \ref{thm:mainintro} follows by taking $q=\infty$.
\begin{theorem}\label{thm:main}
Let $X$ be a separable $2$-smooth Banach space and let $A$ be the generator of an analytic $C_0$-semigroup $(S(t))_{t\geq 0}$ on $X$.
Let $q\in (2, \infty]$, $T>0$, $f\in L^0_{\F}(\Omega;L^q(0,T;\gamma(H,X)))$, set $\alpha = \frac12-\frac1q$
and let $p\in (1, \infty)$ be such that $\alpha\neq \tfrac1p$. Let $N_p(t) = t^{p} \log^{p/2}(t+1)$. Then there exists a constant $C$ such that for all $\delta,\varepsilon>0$
\begin{enumerate}[(i)]
\item $S\diamond f\in B^{\alpha}_{\Phi_2,\infty}(0,T;X)$ a.s.,
\item $(\E\|S\diamond f\|_{B^{\alpha}_{p,\infty}(0,T;X)}^{2p})^{1/(2p)} \leq C  \|f\|_{L^{2p}(\Omega;L^q(0,T;\gamma(H,X)))}$,
\item $\E\|S\diamond f\|_{B^{\alpha}_{\Phi_2,\infty}(0,T;X)}\leq C \|f\|_{L^{\infty}(\Omega;L^q(0,T;\gamma(H,X)))}$,
\item $\|S\diamond f\|_{L^{\Phi_2}(\Omega;B^{\alpha}_{\Phi_2,\infty}(0,T;X))}\leq C\|f\|_{L^{\infty}(\Omega;L^q(0,T;\gamma(H,X)))}$
\item $\P(\|{S\diamond f}\|_{B^{\alpha}_{\Phi_2,\infty}(0,T;X)}>\varepsilon, \,\|f\|_{L^q(0,T;\gamma(H,X))}\le\delta) \leq 2\exp\,\{-C^{-2}\delta^{-2}\varepsilon^2\}$
\item $\|S\diamond f\|_{L^p(\Omega;B^{\alpha}_{\Phi_2,\infty}(0,T;X))}\leq C\|f\|_{L^{N_p}(\Omega;L^q(0,T;\gamma(H,X)))}$.
\end{enumerate}
\end{theorem}
\begin{proof}
Define $Q(t)=\int_0^tS_r\ud r$ and consider the convolution integral
\[
v(t) = \int_0^t S(t-s)(f\cdot W)(s) \ud s,\,\qquad t\in[0,T].
\]
Then $v$ is a continuous adapted process starting from zero and, for every $t\in[0,T]$,
\[
v(t)=\int_0^tQ(t-s)f(s)\ud W(s)
\]
holds a.s.\ by the real stochastic Fubini theorem applied on $\mathbf 1_{[s\le r]}\varphi\circ S(t-r)\circ f(s,\omega)$ where $\varphi\in X^*$ (see e.g.  \cite[Theorem 4.18]{DPZ}). In particular,  $v(t)\in\operatorname{Dom}(A)$ a.s.\  and $S\diamond f=Av+f\cdot W$ a.s.\ for every $t\in[0,T]$. This representation formula is well-known to experts (see \cite[Proposition 4]{DPKZ}). Now we get the result by applying Theorem \ref{cor:detconv2} and Theorem \ref{thm:mainindef}.
\end{proof}

\begin{remark}
Using Remark \ref{rem:othermoments} it is possible to give estimates for other moments than those considered in Theorem \ref{thm:main}.
\end{remark}

\end{document}